\documentclass{amsart}

\usepackage{graphicx}
\usepackage{amssymb}
\usepackage{amsthm}
\usepackage{amsmath}
\usepackage{pstricks,pst-node}
\usepackage{longtable}
\usepackage{array}
\usepackage{ragged2e}
\usepackage{hyperref}
\input xy
\xyoption{all}
\swapnumbers
\newtheorem{theorem}[subsection]{Theorem}
\newtheorem{lemma}[subsection]{Lemma}
\newtheorem{sublemma}[subsection]{Sublemma}

\newtheorem{corollary}[subsection]{Corollary}

\theoremstyle{definition}
\newtheorem{definition}[subsection]{Definition}
\newtheorem{remark}[subsection]{Remark}

\newcommand{\ii}{{\iota}}

\newcommand{\Z}{\mathbb{Z}}
\newcommand{\Q}{\mathbb{Q}}

\newcommand{\sign}{\mathrm{sign}}

\def\G{{\mathcal G}}

\newcommand{\inv}{^{-1}}

               \def \cT {{\mathcal T}}
             \def \cH {{\mathcal H}}

             \def \cF {{\mathcal F}}

             \def \cS {{\mathcal S}}
 
\def \cP {{\mathcal P}}

\def\R{\mathbb{R}}

\def\Q{\mathbb{Q}}
\def\Z{\mathbb{Z}}

\def\inv{^{-1}}

\DeclareMathOperator{\vcd}{vcd}
\DeclareMathOperator{\vol}{vol}

\DeclareMathOperator{\ver}{vert}

\newcommand{\GL}{\mathrm{GL}}
\newcommand{\SL}{\mathrm{SL}}

\usepackage{bm}

\begin{document}

\title[Explicit sharbly cycles at the vcd for $\SL_n(\Z)$]
{Explicit sharbly cycles at the virtual cohomological dimension for $\SL_n(\Z)$}

\author{Avner Ash} \address{Boston College\\ Chestnut Hill, MA 02445}
\email{Avner.Ash@bc.edu} \author{Paul E. Gunnells}
\address{University of Massachusetts Amherst\\ Amherst, MA 01003}
\email{gunnells@umass.edu} \author{Mark McConnell}
\address{Princeton University\\ Princeton, New Jersey 08540}
\email{markwm@princeton.edu}



\begin{abstract}
Denote the virtual cohomological dimension of $\SL_n(\Z)$ by
$t=n(n-1)/2$. Let $St$ denote the Steinberg module of $\SL_n(\Q)$
tensored with $\Q$.  Let $Sh_\bullet\to St$ denote the sharbly
resolution of the Steinberg module.  By Borel-Serre duality, the
one-dimensional $\Q$-vector space $H^0(\SL_n(\Z), \Q)$ is isomorphic
to $H_t(\SL_n(\Z),St)$.  We find an explicit generator of
$H_t(\SL_n(\Z),St)$ in terms of sharbly cycles and cosharbly cocycles.
These methods may extend to other degrees of cohomology of
$\SL_n(\Z)$.  
\end{abstract}

\maketitle

\section{Introduction}\label{intro} The cohomology of $\SL_n(\Z)$ with
rational coefficients is an object of great interest.  Not that much
is known.  The stable cohomology $\varprojlim\limits_{n} H^i(
\SL_n(\Z), \Q)$ was determined by Borel.  More recently, the
restriction of the stable cohomology to $H^*( \SL_n(\Z), \Q)$ for a
fixed $n$ was determined by Franke.  See \cite{GKT} for a good
exposition of Franke's result.  Very recently some other classes,
related to the Borel classes, have been found \cite{FB}, cuspidal
cohomology for $\SL_n(\Z)$ for certain $n$ has been constructed
\cite{bcg}, and many new cohomology classes for $\SL_n(\Z)$ in low
degree but above the stable range have been determined in \cite{GG}.
New cohomology classes for $\SL_n(\Z)$ in very high degree appear
in~\cite{AA}, where the main result of this paper (for $n=3$) is used
in a crucial way.

The theorem of Borel says that $H^i( \SL_n(\Z), \Q)$ equals the stable
cohomology if $i$ is small compared with $n$.  The other end of the
range is where $i$ is near the virtual cohomological dimension (vcd)
$t=n(n-1)/2$ of $\SL_n(\Z)$.  The authors of~\cite{CFP} have
conjectured that $H^{t-i}( \SL_n(\Z), \Q)=0$ if $i<n-1$.  See their
paper for references to proofs of the conjecture for $i=0,1$ all $n$,
and for $i<n-1$ for $n\le 7$.  It has also been proven for $i=2$, all
$n$: see \cite{BMPSW}.

 There have been explicit computations of $H^*( \SL_n(\Z), \Q)$ for
$n\le 12$.  Complete computation for $n=2$ is classical, for $n=3$ is
in \cite{Soule}, for $n=4$ in \cite{LS2} and for $n=5,6,7$ in
\cite{EVGS}.  Partial results for $8\le n\le 11$ are in \cite{SEVKM}.

One way to study  $H^i( \SL_n(\Z), \Q)$ is to use the Borel-Serre isomorphism
\[
H^i( \SL_n(\Z), \Q)\approx H_{t-i}( \SL_n(\Z), St),
\] where 
$St$ denotes the Steinberg module of $\SL_n(\Q)$ tensored with $\Q$.
To study $H_{t-i}( \SL_n(\Z), St)$ we can use the sharbly resolution
of $St$.

In this paper we begin to study the problem of finding explicit
realizations of cohomology classes of $\SL_n(\Z)$ in terms of the
sharbly resolution.  Even the case of $H^0$ is far from obvious.  Let
$G$ be a subgroup of finite index in $\SL_n(\Z)$.  By Borel-Serre
duality \cite{B-S}, the one-dimensional $\Q$-vector space $H^0(G, \Q)$ is
isomorphic to $H_t(G,St)$.  In this paper we find an explicit
generator $z_G$ of $H_t(G,St)$ in terms of sharbly cycles.  To show
that $z_G\ne0$ we use cosharbly cocycles, as defined in
\cite{unstable}.

The constructions and proofs involved have some interesting twists.
To construct the desired sharbly cycles, we use the beautiful theory
of regular triangulations of polytopes, first developed
in~\cite{GKZ}.  The construction of the dual cosharbly cocycles
involves the theory of scissors congruences in Euclidean space, as
found for example in~\cite{dupont} .

 We hope to extend these ideas further to other degrees of cohomology
of $\SL_n(\Z)$, perhaps following the ideas in~\cite{AA}.
 
We would like to thank Bruno Kahn for drawing our attention to
Dupont's work \cite{dupont}.  We also thank the referee for many helpful comments and especially for suggesting we add more detail for the cases $n=4$ and $n=5$.  This led us to discover a gap in our treatment of flips, which has now been fixed.

\section{The Steinberg module and the sharbly resolution}\label{St}

The \emph{Tits building} $T_n$
is the simplicial complex whose vertices are the proper nonzero
subspaces of $\Q^n$ and whose simplices correspond to flags of
subspaces.  By the
Solomon--Tits theorem $T_{n}$ has the homotopy type of a wedge of
$(n-2)$-dimensional spheres.  It is a left $\GL_n(\Q)$-module and
therefore so is its homology.  
\begin{definition}
We define the \emph{Steinberg module} $St$
to be the reduced homology of the Tits building with $\Q$-coefficients:
\[
St=\widetilde H_{n-2}(T_n,\Q).
\]
\end{definition}
Note: the Steinberg module is usually defined as the reduced homology 
with $\Z$-coefficients, but if we tensor that with $\Q$, we obtain what we are here calling $St$.

\begin{definition}\label{sh}
The \emph{Sharbly complex} $Sh_{*} $ is the following
complex of left $\GL_n(\Q)$-modules.   As an $\Q$-vector space,
$Sh_{k}$ is generated by 
symbols $[v_1,\dots,v_{n+k}]$, where the $v_i$ are nonzero column vectors in
$\Q^n$, modulo the submodule generated by the following elements:

(i) $[v_{\sigma
(1)},\dots,v_{\sigma(n+k)}]-\sign (\sigma)[v_1,\dots,v_{n+k}]$ for all
permutations $\sigma$;

(ii) $[v_1,\dots,v_{n+k}]$ if $v_1,\dots,v_{n+k}$ do not span $\Q^n$; and

(iii) $[v_1,\dots,v_{n+k}]-[av_1,v_{2},\dots,v_{n+k}]$ for all $a\in
\Q^\times$.

\noindent
The action of $g\in\GL_n(\Q)$ is 
given by $g[v_1,\dots,v_{n+k}]=[gv_1,\dots,gv_{n+k}]$.

\noindent The boundary map $\partial \colon Sh_{k} \rightarrow
Sh_{k-1}$ is given by 
\[ 
\partial([v_1,\dots,v_{n+k}])=
\sum_{i=1}^{n+k} (-1)^{i+1}[v_1,\dots,\widehat{v_i},\dots v_{n+k}],
\]
where as usual $\widehat{v_i}$ means to delete $v_{i}$.

We call an element of $Sh_k$ a \emph{$k$-sharbly} and an expression of
the form $[v_1,\dots,v_{n+k}]$ a \emph{basic sharbly}, even if
$v_1,\dots,v_{n+k}$ do not span $\Q^n$.  Sharblies originally appeared
in the work of Lee--Szczarba \cite{LS}, hence the name.
\end{definition}

Theorem 5 in \cite{AGM5} immediately implies:

\begin{theorem}\label{Sh}  There is a map of $\GL_n(\Q)$-modules 
 $Sh_0 \to St$
such that 
the following is an exact sequence of $\GL_n(\Q)$-modules:
\[
\cdots\to Sh_k \to Sh_{k-1} \to \cdots \to Sh_0 
\to St \to 0.
\]
\end{theorem}

\begin{theorem}
Let $G$ be a subgroup of finite index in $\SL_n(\Z)$.  Then
$H_k(G, St)$ is isomorphic to the homology at the $k$-th place of the sequence
\[
\cdots \to Sh_{k+1}\otimes_{G} \Q \to
 Sh_k\otimes_{G} \Q \to Sh_{k-1}\otimes_{G} \Q
 \to  \cdots
\]
\end{theorem}
\begin{proof}
Because the stabilizers in $G$ of nonzero basic elements of $Sh_*$ are finite groups, the theorem follows easily from Theorem 7 of 
\cite{AGM5}.
\end{proof}

\begin{definition}
Set $[v_1,\dots,v_{n+k}]_G$ to be the image of $[v_1,\dots,v_{n+k}]$
in the coinvariants  $Sh_k\otimes_{G} \Q$.
\end{definition}

\section{The Voronoi cellulation}\label{Vor}

Let $n\ge2$.  Let $C_n$ be the set of positive definite real $n\times
n$ symmetric matrices.  It is an open cone in the vector space $Y_n$
of all real $n\times n$ symmetric matrices.  For each non-zero
subspace $W$ of $\Q^n$ defined over $\Q$, set $b(W)$ to be the
rational boundary component of $C_n$ consisting of the cone of all
positive semi-definite real $n\times n$ symmetric matrices whose
kernel is $W\otimes\R$.  The minimal Satake bordification $C_n^*$ of
$C_n$ is the union of $C_n$ with all the rational boundary components.
It is convex and hence contractible.  Note that the dimension of $Y_n$
is $n(n+1)/2$, while a rational boundary component of $C_n$ spans a
$\Q$-subspace of $Y_n$ of dimension $k(k+1)/2$ for some $k<n$.

From now on we fix $n$ and set $C=C_n$, $C^{*} = C_{n}^{*}$, and
$Y=Y_n$.

\begin{definition}
If $u$ is any  nonzero column vector in $\Q^n$, let ${}^{t}u$ be its
transpose, and let $u'=u{}^tu\in C^*$.
\end{definition}

We now describe the perfect Voronoi cellulation of $C$.  For more
detail see \cite[II.6]{AMRT} and \cite[Appendix]{G}.  If $v\in \Q^n$
is a nonzero column vector, then $v'$ is a rank 1 matrix in $C^*$, and
thus generates a rational boundary component of dimension 1.  If
$v_1,\dots,v_m$ are $m$ such vectors, we let $s(v_1,\dots,v_m)$ denote
the closed convex conical hull of $v'_1,\dots,v'_m$ in $C^*$.  We call
$v'_1,\dots,v'_m$ the \emph{vertices} of $s(v_1,\dots,v_m)$.  The
vertices are determined uniquely up to scalar multiples.  Their ordering is determined up to an even
permutation by the cone $s(v_1,\dots,v_m)$ together with an
orientation of it.

The perfect Voronoi cellulation of $C^*$ is given by the cells
$s_Q=s(v_1,\dots,v_m)$, where $Q$ runs over all positive definite real
$n\times n$ quadratic forms, and where the nonzero integral vectors
that minimize $Q$ over all integral vectors are exactly $\pm
v_1,\dots,\pm v_m$.  There is a left action of $\SL_n(\Z)$ on $C^*$
given by $\gamma\cdot x=\gamma x{}^t \gamma$.  The Voronoi cellulation
is stable under this action.

There are a finite number of Voronoi cells modulo $\SL (n, \Z)$.  A
Voronoi cell $s(v_1,\dots,v_{k})$ lies in a boundary component of
$X_n^*$ if and only if $v_1,\dots,v_{k}$ do not span $\Q^n$.

 This cellulation is called the perfect Voronoi cellulation for
the following reason.  A positive definite quadratic form with minimal
vectors $\pm v_1,\dots,\pm v_m$ is called \emph{perfect} if and only if $
v'_1,\dots, v'_m$ span the $\R$-vector space $Y$.  The top-dimensional
cones in the perfect Voronoi cellulation are the $s_Q$, where $Q$ runs
over all perfect forms.  

Let $G$ be a subgroup of finite index in $\SL_n(\Z)$, $R$ 
a set of
representatives of $G$-orbits of perfect forms, and
\[
U=\bigcup_{Q\in R} s_{Q}.
\]
If $G$ is torsionfree, then $U$ is a fundamental domain for the $G$-action
on $C^*$.  If $G$ is not torsionfree, the projection of $U$ to
$G\backslash C^*$ is still surjective, and the stabilizers in $G$ of
the cones $s_Q$ in $U$ are finite groups.

\section{Construction of the sharbly cycle: $G$ torsionfree}\label{cycle1}

Let $G$ be a subgroup of finite index in $\SL_n(\Z)$, let
$t=n(n-1)/2=\vcd(G)$, and $d=n(n+1)/2=n+t=\dim C$.  In this section and the next,
we construct sharbly cycles representing classes in 
$H_t(G, St)$.  In Section~\ref{cococycle} we will show that each of these classes are nonzero,
and therefore each one generates $H_t(G, St)\approx H^0(G, \Q)\approx \Q$.

In what follows we will need definitions from the Appendix and Theorem~\ref{glue}, so the reader may wish to read the Appendix before proceeding.
Fix compatible orientations on $C$ and $Y$. 

\begin{definition} \label{def1}\ 

\noindent $\bullet$  A \emph{tile} is an oriented top-dimensional cone in the perfect Voronoi decomposition of $C^*$, or its image in $G\backslash C^*$.

\noindent $\bullet$ A  \emph{facet} is an unoriented codimension 1 cone which is a face of a tile.

\noindent $\bullet$ Let $\cT$ be the set of all tiles $T$ where $T$ is given
the orientation induced from $Y$.  Let $\cF$ be the set of all facets $F$.

\noindent $\bullet$
For $T\in\cT$ let
$\Sigma(T)$ denote the set of all top-dimensional simplicial cones $s$ whose vertices are a subset of the vertices of $T$,  and where $s$ is given
the orientation induced from $C$.  

\noindent $\bullet$
For $F\in\cF$ let
$\Sigma(F)$ denote the set of all simplicial cones $s$ whose vertices are a subset of the vertices of $F$ and whose dimension equals the dimension of $F$.

\noindent $\bullet$
Let $X\in\cT$ or $\cF$.
A \emph{triangulation} of $X$ is a decomposition of $X$ into a collection of elements $s_i\in\Sigma(X)$ such that for $i\ne j$, $s_i\cap s_j$ is either empty or a common face of both $s_i,s_j$.  A \emph{regular triangulation} of $X$ is defined in the Appendix.

\noindent $\bullet$
If $U$ is a union of elements of $\cT$, a regular triangulation of $U$ is a regular triangulation of each $T\in U$. 

\noindent $\bullet$
Given a simplicial cone $s$ in $C^*$ whose vertices are $v'_i$ , and given an ordering $v_1,\dots,v_m$, set 
$s(v_1,\dots,v_m)=s$ and 
$[s]=[v_1,\dots,v_m]$ (a basic sharbly).  

\end{definition}
In the last bullet, if $s$ is a top-dimensional oriented simplicial cone, so that $m=d$,  then unless otherwise specified, we assume that $v_1,\dots,v_m$ are written in an order which induces the orientation on $s$ that is compatible with the fixed orientation of $Y$.
This is well-defined, because each $v_i$ is determined up to a scalar
multiple, and the order of the vertices is defined up to an even
permutation.  (See (i) and (iii) in Definition \ref{sh}.)

Let $G$  be a torsionfree subgroup of finite index in $\SL_n(\Z)$.  Let $\cS=\{s\}$ be
a set of top-dimensional oriented simplicial cones  whose
union is a fundamental domain for $G$ acting on $C^*$, obtained as follows:
 take a set of representatives $R_G$ of $G$-orbits of $\cT$, and for $U=\cup_{T\in R_G} T$, choose a regular triangulation of $U$, and let $\cS$ be the set of all the simplicial cones in the triangulation, oriented with the orientation induced by $Y$.  There is no reason that, in general, if $T_1,T_2\in R_G$, $g\in G$ and $gT_1, T_2$  meet in a facet $F$, that 
the triangulations on $F$ induced by $gT_1$ and $T_2$ should match, not even if $g=1$.
  
 \begin{definition}\label{flip}
   A \emph{flipon} is a basic $t$-sharbly $[v_1,\dots,v_d]$ such that 
 there is an affine subspace of $Y$ of dimension $d-2$ that contains $v_i'$ for all $i=1,\dots,d$. Its image in the $G$-invariants, $[v_1,\dots,v_d]_G$, is also called a flipon.
 \end{definition}

\begin{theorem}\label{tf}
 Let $G$ be torsionfree and choose $\cS$ as above.  Then there exist flipons 
 $[y^\alpha_1,\dots,y^\alpha_d]$ such that 
$\partial z_G = 0$, 
where \[z_G=\sum _{ s \in \cS} [s]_G+\sum_\alpha[y^\alpha_1,\dots,y^\alpha_d].\]
\end{theorem}

 Note that $z_G$ depends on a number of choices,  but we suppress that in the notation.
 
\begin{proof}
Begin by writing $s=s(w_1,\dots,w_d)$, where $w_1,\dots,w_d$ depend on $s$, and define
\[
\Phi:=\partial \sum _{ s \in \cS} [s]_G=\sum_{s=s(w_1,\dots,w_d)\in\cS} 
\sum_{i=1}^d (-1)^{i+1}[w_1,\dots,\widehat w_i, \dots, w_d]_G.
\] 
Note that our chosen orientation on $Y$ induces on the cone
$s(w_1,\dots,\widehat w_i, \dots, w_d)$ the orientation given by the
ordering of its vertices as written, times $(-1)^{i+1}$.

We need to show that $\Phi=\sum\partial[v]_G$ for some finite set of flipons $[v]_G$.
Break up the sum defining $\Phi$ by the tiles in $R_G$:

 \begin{equation}\label{phi}
\Phi:=\sum_{T\in R_G}  \sum _{ s=s(w_1,\dots,w_d) \in \cS, s\subset T} 
\sum_{i=1}^d (-1)^{i+1}[w_1,\dots,\widehat w_i, \dots, w_d]_G.
  \end{equation}
  
The boundary $C^*$ has codimension at least 2.  
Each facet $F$ of a tile has codimension 1 in $C$, so no facet can be contained in the boundary of $C^*$. 
Also, since $G$ is torsionfree, no element of $G$ can stabilize a facet, except for the identity.  The same goes for each of the $s(w_1,\dots,\widehat w_i, \dots, w_d)$. 

Suppose the interior of $u=s(w_1,\dots,\widehat w_i, \dots, w_d)$ is not contained in any facet. Then the interior of $u$ is contained in the interior of some tile $T$ and in the boundary of 
one of the top-dimensional simplicial cones in the triangulation of $T$.  There must be exactly one other $(t-1)$-sharbly whose associated cone also equals $u$, appearing in the sum 
defining $\Phi$, coming from the boundary of a different  top-dimensional simplicial  cone in the triangulation, also contained in $T$.  Since the orientations on $u$ induced by these two simplicial cones are opposite,  the two corresponding $(t-1)$-sharblies cancel out in the sum defining $\Phi$.

We can write the remaining terms in the sum arranged by facets.  Let
$F_G$ be the set of facets of tiles in $R_G$.  
 If $u=s(x_1,\dots,x_{d-1})$ is in the boundary of $s$, write $[u]^s$
for the basic sharbly $[x_1,\dots,x_{d-1}]$ where the $x_i$ have been
placed in an order that determines the orientation of $u$ induced from
$s$.  (Remember that we give all $s$'s the orientation induced from
$Y$.)

Therefore

 \begin{equation}\label{phi2}
\Phi=\sum_{F\in F_G}  \sum _{s\in \cS, u \in \partial s, u\subset F} 
[u]^s_G.
  \end{equation}
  We must show that $\Phi$ is a sum of boundaries of flipons.
  
  The tiles modulo $G$ provide a cellular decomposition of  $G\backslash C^*$.  
Because $G$ is torsionfree,  $G\backslash C$ is a manifold.  It follows that in $G\backslash C^*$, any facet $E$ of a tile is in the boundary of exactly two tiles and  the orientations on  the facet induced from  the two tiles
 are opposite to each other.
 
 Any $s$ appearing in the sum in~\eqref{phi2} is contained in a unique facet $E$.  Such a facet satisfies $E=T_1\cap gT_2$, with $T_1\ne gT_2$ for some $g\in G$, 
 $T_1, T_2\in R_G$.  Because $G$ is torsionfree, $g$ here is uniquely determined by $E$. 
  
Given $E\in F_G$, let $E=T_1\cap gT_2$ as above, and set
  \begin{equation}\label{phi3}
\Phi_E:=\sum_{s\in\cS, s\subset T_1, u\subset E\cap \partial s} [u]_G +
 \sum_{s\in\cS, s\subset gT_2, u\subset E\cap \partial s} [u]_G.
  \end{equation}
  
Then $\Phi=\sum_{E\in F_G}\Phi_E$.  The set $\{u\}$ in the first sum
constitute an oriented regular triangulation of $E$ and the set
$\{u\}$ in the second sum constitute a (perhaps different) oriented
regular triangulation of $E$.


 We proceed to show that 
$\Phi=\sum_{E\in F_G}\Phi_E$
is a sum of boundaries of flipons.  
It is not true that $\Phi_E$ individually is  a sum of boundaries of flipons if the corresponding flip has a non-maximal circuit.

Anticipating the notation we will use in Section~\ref{cococycle}, we have
the hyperplane $H$ in $Y$ that cuts $C$ transversally.  If $v\in\Q^n$, $v'$ is the rank 1 matrix $vv^t\in C^*$, and $v''$ is the element of $H$ which is in the ray from 0 through to $v'$.
Let $c(v_1,\dots,v_r)$ denote the convex hull of 
$v''_1,\dots,v''_r$.  Circuits and flips are defined in the Appendix.
Call $v_1,\dots,v_t\in\Q^n$ a circuit if and only if $v''_1,\dots,v''_t$ is a circuit in $H$.
  
  Let $A(E)$ index the sharblies contained in $\Phi_E$.
It follows from Corollary~\ref{glue2} that there exists 
  flipons $[v^\alpha_1,\dots,v^\alpha_d]_G$ such that 
  $(v^\alpha_1)'',\dots,(v^\alpha_{p(\alpha)})''$ is a circuit and 
 \[
 \Phi_E = 
 \sum_{\alpha\in A(E)}
\sum_{i=1}^{p(\alpha)} (-1)^i [v^{\alpha}_1,\dots, \widehat {v^{\alpha}_i},\dots, v^{\alpha}_{d}]_G.
 \]
 Here, 
  $(v^\alpha_1)'',\dots,(v^\alpha_{p(\alpha)})''$ is the circuit of the corresponding flip.

The theorem  follows from:
 \begin{lemma}
 Let $A=\cup_E A(E)$.
 There exist flipons  $[w^{\beta}_1,\dots, w^{\beta}_{d}]_G$ such that
\[
\sum_E \Phi_E = \partial(
 \sum_{\alpha\in A} [v^{\alpha}_1,\dots, v^{\alpha}_{d}]_G-
\sum_{\beta}
 [w^{\beta}_1,\dots, w^{\beta}_{d}]_G).
 \]
 \end{lemma}
 We call the flipons $[w^{\beta}_1,\dots, w^{\beta}_{d}]_G$ ``secondary flipons''.
 \begin{proof}
Note that
\[
\partial
 [v^{\alpha}_1,\dots, v^{\alpha}_{d}]_G=
\sum_{i=1}^{d} (-1)^i [v^{\alpha}_1,\dots, 
\widehat{v^{\alpha}_{i}},\dots,
v^{\alpha}_{d}]_G.
\]
Define:
\[
\Psi_G=
 \sum_A 
 \sum_{j=p(\alpha)+1}^d
 (-1)^j
[v^{\alpha}_1,\dots, v^{\alpha}_{p(\alpha)}, 
v^{\alpha}_{p(\alpha)+1}, \dots, \widehat{v^{\alpha}_{j}},\dots,
v^{\alpha}_{d}]_G.
\]
Then
\[
\Psi_G=
\partial(\sum_A
 [v^{\alpha}_1,\dots, v^{\alpha}_{d}]_G))
 -\sum_E \Phi_E.
 \]
 
 \begin{remark}
 We can think of $\Psi_G$ as an error term.  In the example of the pyramid in the appendix, the error term would be the square $[1234]$.  There is no way to get rid of these error terms one at a time, but their sum will be the boundary of a sum of secondary flipons.
 \end{remark}
 
 We will show there exist flipons  
$[w^{\beta}_1,\dots, w^{\beta}_{d}]_G$ such that
\[\Psi_G=
\partial
\sum_{\beta}
 [w^{\beta}_1,\dots, w^{\beta}_{d}]_G.
\]

Since $\Phi=\sum_E \Phi_E$ is itself a boundary, 
$\partial\Psi_G=0$.  So there is a chance that $\Psi_G$ is indeed a boundary of the type desired.

$\Psi_G$ is a sum of sharblies 
\[
\pm
[v^{\alpha}_1,\dots, v^{\alpha}_{p(\alpha)}, 
v^{\alpha}_{p(\alpha)+1}, \dots, \widehat{v^{\alpha}_{j}},\dots,
v^{\alpha}_{d}].
\]
For brevity, fix  $\alpha$ and $j>p(\alpha)$ and write
\[
\eta=(-1)^j[v^\alpha_1, \dots, \widehat{v^\alpha_j},\dots,
v^\alpha_d],
\]
So $\Psi_G=\sum \eta$ and $\partial\Psi_G=0$.

Let's collect some facts about $\eta$.

\noindent$\bullet$ $\eta$ contains a circuit.

\noindent
Indeed, it contains the circuit 
$v^{\alpha}_1,\dots, v^{\alpha}_{p(\alpha)}$.  

\noindent$\bullet$ Let $c(\eta)=
c(v^\alpha_1, \dots, \widehat{v^\alpha_j},\dots,v^\alpha_d)$.  Then $c(\eta)$ has dimension exactly $d-3$.  
\noindent
Why?
A tile has dimension $d-1$.  A facet has dimension $d-2$,  
so each simplex in it, such as 
$c(v^{\alpha}_2,\dots, v^{\alpha}_{d})$,
 has dimension $d-2$.  
Since $v^\alpha_1,\dots, v^\alpha_{p(\alpha)}$ is a circuit, the dimension does not change when we remove $v^\alpha_1$.
So the dimension of
$c(v^{\alpha}_1,\dots, v^{\alpha}_{d})$ is $d-2$.  When we remove $v^\alpha_j$, the dimension either stays $d-2$ or goes down to $d-3$.  But it cannot stay $d-2$ because there are $d-1$ vectors in 
$(v^\alpha_1, \dots, \widehat{v^\alpha_j},\dots,v^\alpha_d)$, and at least one circuit, so its dimension is at most $d-2-1$. 

\noindent$\bullet$ $\eta$ contains a unique circuit.
\noindent
There cannot be two circuits contained in 
$v^{\alpha}_1,\dots, v^{\alpha}_{d}$.
This is because each circuit causes the dimension of
 $c(v^{\alpha}_1,\dots, v^{\alpha}_{d})$ to go down by 1 from what would be the case if 
$v^{\alpha}_1,\dots, v^{\alpha}_{d}$ were affinely independent.  
If $v^{\alpha}_1,\dots, v^{\alpha}_{d}$ were affinely independent then 
 $c(v^{\alpha}_1,\dots, v^{\alpha}_{d})$  would have dimension $d-1$.  So if there were two or more circuits, 
then the dimension of  $c(v^{\alpha}_1,\dots, v^{\alpha}_{d})$  would be $d-3$ or less, but it has dimension $d-2$.

We want to prove that  
$\Psi_G=\partial N$, where $N$ is a sum of flipons.
We have $\partial\Psi_G=0$ and
\[
\Psi_G=
 \sum_A 
 \sum_{j=p(\alpha)+1}^d
 (-1)^j
[v^{\alpha}_1,\dots, v^{\alpha}_{p(\alpha)}, 
v^{\alpha}_{p(\alpha)+1}, \dots, \widehat{v^{\alpha}_{j}},\dots,
v^{\alpha}_{d}]_G.
\]
Then $\partial\Psi_G=I+II+III$ where
\[
I=
 \sum_A 
 \sum_{j=p(\alpha)+1}^d
 (-1)^j
  \sum_{i=1}^{p(\alpha)}
 (-1)^i
[v^{\alpha}_1,
\dots, \widehat{v^{\alpha}_{i}},\dots,
v^{\alpha}_{p(\alpha)}, 
v^{\alpha}_{p(\alpha)+1}, \dots, \widehat{v^{\alpha}_{j}},\dots,
v^{\alpha}_{d}]_G;
\]
\[
II=
 \sum_A 
 \sum_{j=p(\alpha)+2}^d
 (-1)^j
  \sum_{i=p(\alpha)+1}^{j-1}
 (-1)^i
[v^{\alpha}_1,
\dots, 
v^{\alpha}_{p(\alpha)}, 
v^{\alpha}_{p(\alpha)+1}, \dots, 
\widehat{v^{\alpha}_{i}},\dots,
\widehat{v^{\alpha}_{j}},\dots,
v^{\alpha}_{d}]_G;
\]
\[
III=
 \sum_A 
 \sum_{j=p(\alpha)+1}^{d-1}
 (-1)^j
  \sum_{i=j+1}^{d}
 (-1)^{i-1}
[v^{\alpha}_1,
\dots, 
v^{\alpha}_{p(\alpha)}, 
v^{\alpha}_{p(\alpha)+1}, \dots, 
\widehat{v^{\alpha}_{j}},\dots,
\widehat{v^{\alpha}_{i}},\dots,
v^{\alpha}_{d}]_G.
\]
It is easy to see that $II+III=0$.  This is the usual fact that the boundary of a boundary is 0.  For an explicit proof in this case, see the proof of Lemma~\ref{V} below, and just erase $x$ from it everywhere.

We conclude:
\begin{lemma}\label{dpsi}
\[
I=
 \sum_A 
 \sum_{j=p(\alpha)+1}^d
 (-1)^j
  \sum_{i=1}^{p(\alpha)}
 (-1)^i
[v^{\alpha}_1,
\dots, \widehat{v^{\alpha}_{i}},\dots,
v^{\alpha}_{p(\alpha)}, 
v^{\alpha}_{p(\alpha)+1}, \dots, \widehat{v^{\alpha}_{j}},\dots,
v^{\alpha}_{d}]_G=0.
\]
\end{lemma}
  
We are going to cone off.  
First we have to rigidify $\Psi_G$.  
\begin{definition}
(1) Let $S$ be a subset of $C^*$ that is contained in some face of the Voronoi tessellation.  Let $f(S)$ denote the intersection of all the  faces of the Voronoi tessellation that contain $S$.  This is the \emph{minimal face} containing $S$.

(2)  $\ver(f(S))$ denotes the set of vertices of $f(S)$.  

(3) If $v_1, \dots,v_r\in \Q^n-\{0\}$, by abuse of terminology, we may refer to
$f(\{v_1'',\dots,v_r''\})$ as  the minimal face of $v_1,\dots,v_r$.

(4) $f(\alpha,j)$ denotes the minimal face of $\{v^{\alpha}_1,\dots,  \widehat{v^{\alpha}_{j}},\dots,v^{\alpha}_{d}\}$.

(5) $\cF$ is a set of representatives of the $G$-orbits of the set of  faces of the Voronoi tessellation.
\end{definition}

Note that $f(S)$ is a face of the Voronoi tessellation, and is uniquely determined by $S$.  So the definition makes sense, and $\ver(f(S))$ is a finite set.

\begin{definition}
For each $\alpha$ and $j>p(\alpha)$, within the $G$-orbit of 
the ordered tuple $( v^{\alpha}_1,\dots,  \widehat{ v^{\alpha}_{j}},\dots, v^{\alpha}_{d})$, 
let $( \tilde v^{\alpha}_1,\dots,  \widehat{ \tilde v^{\alpha}_{j}},\dots, \tilde v^{\alpha}_{d})$
be the element such that 
\[
f(\{\tilde v^{\alpha}_1,\dots,  \widehat{\tilde v^{\alpha}_{j}},\dots,\tilde v^{\alpha}_{d}\})\in\cF.
 \]
\end{definition}

For this to be weill-defined, we have to show that $( \tilde v^{\alpha}_1,\dots,  \widehat{ \tilde v^{\alpha}_{j}},\dots, \tilde v^{\alpha}_{d})$ is unique, in other words, if 
$f(g( v^{\alpha}_1,\dots,  \widehat{v^{\alpha}_{j}},\dots, v^{\alpha}_{d}))
=f( v^{\alpha}_1,\dots,  \widehat{v^{\alpha}_{j}},\dots, v^{\alpha}_{d})$,
we need to show that $g=1$.  Let $f=f( v^{\alpha}_1,\dots,  \widehat{v^{\alpha}_{j}},\dots, v^{\alpha}_{d})$.
Since the minimal face is unique,
$f(g( v^{\alpha}_1,\dots,  \widehat{v^{\alpha}_{j}},\dots, v^{\alpha}_{d})
=gf$.
So assume that $gf=f$. Since $c( v^{\alpha}_1,\dots,  \widehat{v^{\alpha}_{j}},\dots, v^{\alpha}_{d})$ has dimension $d-2$, it meets the interior of the cone $C$,
 and so does the minimal face $f$ containing it.  
 The stabilizer in $G$ of a face of the Voronoi tessellation that meets the interior of $C$ is finite, and hence trivial, since $G$ is torsion-free.
 So $g=1$.

Note that $[\tilde v^{\alpha}_1,\dots,  \widehat{ \tilde v^{\alpha}_{j}},\dots, \tilde v^{\alpha}_{d}]_G=[v^{\alpha}_1,\dots,  \widehat{ v^{\alpha}_{j}},\dots, v^{\alpha}_{d}]_G$.
 Now that we have chosen a distinguished representative from each $G$-orbit, we will remove the tildes and call it
$(v^{\alpha}_1,\dots,  \widehat{v^{\alpha}_{j}},\dots,v^{\alpha}_{d})$.

 \begin{definition}\label{juniv}
   A \emph{universal} sharbly is a sum of basic sharblies 
   $[v_1,\dots,v_r]$ before we take $G$-invariants.
  \end{definition}
  
We define the universal sharbly
\[
\Psi=
 \sum_A 
 \sum_{j=p(\alpha)+1}^d
 (-1)^j
[v^{\alpha}_1,\dots, v^{\alpha}_{p(\alpha)}, 
v^{\alpha}_{p(\alpha)+1}, \dots, \widehat{v^{\alpha}_{j}},\dots,
v^{\alpha}_{d}].
\]
Let $\Psi_G$ denote the image of $\Psi$ in the $G$-coinvariants.
We have not changed the identity of 
$\Psi_G$.  It is still equal to what it was before.

Next,
let $x\in\Z^n-\{0\}$ be chosen arbitrarily and fixed.  
Define the universal sharbly
\[
\Omega=
 \sum_A 
 \sum_{j=p(\alpha)+1}^d
 (-1)^j
[x,v^{\alpha}_1,\dots, v^{\alpha}_{p(\alpha)}, 
v^{\alpha}_{p(\alpha)+1}, \dots, \widehat{v^{\alpha}_{j}},\dots,
v^{\alpha}_{d}].
\]
Note that each summand in $\Omega$ is a flipon.  That is because 
\[c(v^{\alpha}_1,\dots, v^{\alpha}_{p(\alpha)}, 
v^{\alpha}_{p(\alpha)+1}, \dots, \widehat{v^{\alpha}_{j}},\dots,
v^{\alpha}_{d})\] has dimension $d-3$ as proved above, and when we throw in $x$ the dimension  either stays the same or only goes up to $d-2$.
(\footnote{
It doesn't matter that sometimes $x$ may also appear as one of the $v^\alpha_r$'s.
})

Let $\Omega_G$ denote the image of $\Omega$ in the $G$-coinvariants.  Then
\[
\partial\Omega_G=\Psi_G + [x,I]_G + [x,II]_G + [x,III]_G
\] 
with the obvious notation.  
The proof of the theorem will be completed when we show that 
\[
[x,I]_G + [x,II]_G + [x,III]_G=0.
\]
  
\begin{lemma}\label{V}
 $[x,II]_G + [x,III]_G=0$.  
\end{lemma}

\begin{proof}
\begin{gather*}
[x,II]_G + [x,III]_G=\\
 \sum_A 
 \sum_{j=p(\alpha)+2}^d
  \sum_{i=p(\alpha)+1}^{j-1}
 (-1)^{i+j}
[x,v^{\alpha}_1,
\dots, 
v^{\alpha}_{p(\alpha)}, 
v^{\alpha}_{p(\alpha)+1}, \dots, 
\widehat{v^{\alpha}_{i}},\dots,
\widehat{v^{\alpha}_{j}},\dots,
v^{\alpha}_{d}]_G\ +\\
 \sum_A 
 \sum_{j=p(\alpha)+1}^{d-1}
  \sum_{i=j+1}^{d}
 (-1)^{i+j-1}
[x,v^{\alpha}_1,
\dots, 
v^{\alpha}_{p(\alpha)}, 
v^{\alpha}_{p(\alpha)+1}, \dots, 
\widehat{v^{\alpha}_{j}},\dots,
\widehat{v^{\alpha}_{i}},\dots,
v^{\alpha}_{d}]_G.
\end{gather*}

This is the sum over $\alpha\in A$ of terms (where we change the order of summation in the second sum)  
\begin{gather*}
 \sum_{j=p(\alpha)+2}^d
  \sum_{i=p(\alpha)+1}^{j-1}
 (-1)^{i+j}
[x,v^{\alpha}_1,
\dots, 
v^{\alpha}_{p(\alpha)}, 
v^{\alpha}_{p(\alpha)+1}, \dots, 
\widehat{v^{\alpha}_{i}},\dots,
\widehat{v^{\alpha}_{j}},\dots,
v^{\alpha}_{d}]_G\ +\\
 \sum_{i=p(\alpha)+2}^{d}
  \sum_{j=p(\alpha)+1}^{i-1}
 (-1)^{i+j-1}
[x,v^{\alpha}_1,
\dots, 
v^{\alpha}_{p(\alpha)}, 
v^{\alpha}_{p(\alpha)+1}, \dots, 
\widehat{v^{\alpha}_{j}},\dots,
\widehat{v^{\alpha}_{i}},\dots,
v^{\alpha}_{d}]_G,
\end{gather*}
and each of these equals 0
since we can switch the dummy variables $i$ and $j$ in the second sum.
\end{proof}

So we will be all finished when we prove:
\begin{lemma}\label{final}
$[x,I]_G=0$.
\end{lemma}

\begin{proof}
We know that
\[
I=
 \sum_A 
 \sum_{j=p(\alpha)+1}^d
 (-1)^j
  \sum_{i=1}^{p(\alpha)}
 (-1)^i
[v^{\alpha}_1,
\dots, \widehat{v^{\alpha}_{i}},\dots,
v^{\alpha}_{p(\alpha)}, 
v^{\alpha}_{p(\alpha)+1}, \dots, \widehat{v^{\alpha}_{j}},\dots,
v^{\alpha}_{d}]_G=0.
\]
  
Now
\[
[x,I]_G=
 \sum_A 
 \sum_{j=p(\alpha)+1}^d
 (-1)^j
  \sum_{i=1}^{p(\alpha)}
 (-1)^i
[x, v^{\alpha}_1,
\dots, \widehat{v^{\alpha}_{i}},\dots,
v^{\alpha}_{p(\alpha)}, 
v^{\alpha}_{p(\alpha)+1}, \dots, \widehat{v^{\alpha}_{j}},\dots,
v^{\alpha}_{d}]_G.
\]
We want to show $[x,I]_G=0$.
Why is $I=0$?  If we look at the universal sharbly
\[
I_{univ}=
 \sum_A 
 \sum_{j=p(\alpha)+1}^d
 (-1)^j
  \sum_{i=1}^{p(\alpha)}
 (-1)^i
[v^{\alpha}_1,
\dots, \widehat{v^{\alpha}_{i}},\dots,
v^{\alpha}_{p(\alpha)}, 
v^{\alpha}_{p(\alpha)+1}, \dots, \widehat{v^{\alpha}_{j}},\dots,
v^{\alpha}_{d}]
\]
and ask why its image in the $G$-coinvariants is 0, the answer is that its terms must cancel out in pairs (including the possibility that a term could pair with itself.)  For each such pair of terms $m,n$ there is a $g\in G$ such that $gm+n=0$.  If  the same $x$ is assigned to $m$ and $n$ and if it is fixed by $g$, then the same cancelation will occur in $[x,I]_G$.

So we will be finished when we prove the following:
\begin{sublemma}
Let $\alpha,\beta\in A$, 
$1\le i \le p(\alpha)$, $1\le k \le p(\beta)$, $j>p(\alpha)$ and 
 $\ell>p(\beta)$.
 
 (1) The minimal face $f(v^{\alpha}_1,
\dots, \widehat{v^{\alpha}_{i}},
 \dots, \widehat{v^{\alpha}_{j}},\dots,
v^{\alpha}_{d})=f(\alpha,j)$.
 
 (2)
If
\[
g[v^{\alpha}_1,
\dots, \widehat{v^{\alpha}_{i}},
 \dots, \widehat{v^{\alpha}_{j}},\dots,
v^{\alpha}_{d}]
=
[v^{\beta}_1,
\dots, \widehat{v^{\beta}_{k}},
\dots, \widehat{v^{\beta}_{\ell}},\dots,
v^{\beta}_{d}]
\]
for some $g\in G$, then  $gx=x$.
\end{sublemma}

\begin{proof}
(1)   Since $(v^\alpha_1)'',\dots,(v^\alpha_d)''$ are all in some facet $F$, they are all in   some tile $T$.
We want to show
\[
f(v^{\alpha}_1,
\dots, \widehat{v^{\alpha}_{i}},
 \dots, \widehat{v^{\alpha}_{j}},\dots,
v^{\alpha}_{d})=f(v^{\alpha}_1,
 \dots, \widehat{v^{\alpha}_{j}},\dots,
v^{\alpha}_{d}).
\]
The minimal face containing $S$ is the intersection of all the faces of the Voronoi tessellation that contain $S$.
The inclusion $\subset$ is because every face of $T$ that contains 
$(v^{\alpha}_1,
 \dots, \widehat{v^{\alpha}_{j}},\dots,
v^{\alpha}_{d})$ also contains
$(v^{\alpha}_1,
\dots, \widehat{v^{\alpha}_{i}},
 \dots, \widehat{v^{\alpha}_{j}},\dots,
v^{\alpha}_{d})$.

To show  the inclusion $\supset$, it is enough to show that 
if $\phi$ is any face of $T$ that contains 
$(v^{\alpha}_1)'',
\dots, \widehat{(v^{\alpha}_{i})''},
 \dots, \widehat{(v^{\alpha}_{j})''},\dots,
(v^{\alpha}_{d})''$
then $\phi$ also contains $(v^{\alpha}_i)''$.   

Let $\phi$ be such a face.  Then $\phi=T\cap\cH$ for some supporting affine hyperplane $\cH$.  We already know that $(v^{\alpha}_i)''\in T$.  Meanwhile, 
$(v^{\alpha}_1)'',\dots,(v^{\alpha}_i)'',\dots,(v^{\alpha}_{p(\alpha)})''$ is a circuit.  So $(v^{\alpha}_i)''$ is contained in the affine span of 
$(v^{\alpha}_1)'',\dots,
\widehat{(v^{\alpha}_i)''},\dots,(v^{\alpha}_{p(\alpha)})''$.
Since $\cH$ contains
$(v^{\alpha}_1)'',
\dots, \widehat{(v^{\alpha}_{i})''},
 \dots, \widehat{(v^{\alpha}_{j})''},\dots,
(v^{\alpha}_{d})''$,
in particular it contains 
$(v^{\alpha}_1)'',\dots,
\widehat{(v^{\alpha}_i)''},\dots,(v^{\alpha}_{p(\alpha)})''$, and therefore
it also contains $(v^{\alpha}_i)''$.  We conclude that $(v^{\alpha}_i)''\in T\cap\cH=\phi$.

(2) From the hypothesis, taking minimal faces,
 \[
g(f(v^{\alpha}_1,
\dots, \widehat{v^{\alpha}_{i}},
 \dots, \widehat{v^{\alpha}_{j}},\dots,
v^{\alpha}_{d}))
=
f(v^{\beta}_1,
\dots, \widehat{v^{\beta}_{k}},
\dots, \widehat{v^{\beta}_{\ell}},\dots,
v^{\beta}_d).
\]
By (1),
\[
gf(\alpha,j)=f(\beta,\ell).
\]
But $g$ acts freely on faces, and both $f(\alpha,j), f(\beta,\ell)$ are in 
$\cF$, the set of representatives of $G$-orbits of faces.  Therefore $g=1$ and $gx=x$.
\end{proof}
\end{proof}
 \end{proof}
\end{proof}

\section{Construction of the sharbly cycle: general $G$ }\label{cycle2}

Now let $G$ be any subgroup of finite index in $\SL_n(\Z)$.  Let $D$ be a torsionfree normal subgroup of
$G$ of finite index.  We have defined $z_D$ above, and proved that 
$\partial z_D=0$.  Our definition of $z_D$ depended on the choice of a triangulated fundamental domain for $D$ in $C^*$, which we take to be fixed.

\begin{definition} Define $\ii:Sh_D\to Sh_G$ by the 
formula $\ii([\bullet]_D)=[\bullet]_G$.  
\end{definition}

It is easy to check that $\ii$ is well-defined and commutes with taking boundary.  
(If we apply Borel-Serre duality, $\ii$ corresponds to the transfer map on cohomology.)

\begin{definition}
$z_G=[G:D]\inv\ii(z_D)$.
\end{definition}

In this definition there are implicit dependencies on the choice of $D$, the choice of regular triangulation of a fundamental domain $U_D$ for $D$ on the Voronoi cellulation, and a choice of flipons.  

\begin{theorem}
$\partial z_G=0$.
\end{theorem}

\begin{proof}
By Theorem~\ref{tf},
 $\partial z_D = 0$.
Hence $\partial  z_{G}=[G:D]\inv\ii\partial(z_D)=0$.
\end{proof}

When all the tiles are simplicial cones, then there is no choice of
triangulation and no need for flipons.  This happens when $n=2,3,4$.  In
addition:

\begin{theorem}
Let $n=2$ or $3$.  Then $z_{\SL_n(\Z)}$ is independent of the choice of $D$.
\end{theorem}

\begin{proof}
Let $n=2$ or $3$.  Then there is one $\SL_n(\Z)$-orbit of tile, represented by $T_0$, say, which is a simplicial cone.  
Let $H_0$ denote the stabilizer of $T_0$ in $\SL_n(\Z)$.
Let $D$ be a torsionfree normal subgroup of finite index in $\SL_n(\Z)$.  Let $\G$ be a set of representatives of the double cosets $D\backslash\SL_n(\Z)/H_0$.  Then
\[
z_D=\sum_{h\in\G} [hT_0]_D
\]
and
\[
z_G=[G:D]\inv\sum_{h\in\G} [hT_0]_{\SL_n(\Z)}=
[G:D]\inv\sum_{h\in\G} [T_0]_{\SL_n(\Z)}=
[G:D]\inv|\G| [T_0]_{\SL_n(\Z)}.
\]
Since $D$ is torsionfree,  $H_0$ injects into the finite group $D\backslash\SL_n(\Z)$.
So $z_G=|H_0|\inv [T_0]_{\SL_n(\Z)}$, independent of the choice of $D$.
\end{proof}
 
 \section{Construction of the cosharbly cocycle}\label{cococycle}

\begin{definition} A \emph{$t$-cosharbly} for $G\subset\SL_n(\Z)$ is 
a $G$-invariant linear functional $\mu: St_t\to E$
for some trivial $\Q G$-module $E$ such that $\mu$ vanishes on the
elements enumerated in (i), (ii), and (iii) of Definition~\ref{sh}.
It is a \emph{$t$-cosharbly cocycle} 
if it vanishes on $\partial([v_1,\dots,v_{d+1}])$ for
all nonzero $v_1,\dots,v_{d+1}\in \Q^n$.
\end{definition}

Note that the bilinear pairing between $t$-cosharblies and $t$-sharblies descends to a well-defined pairing between $t$-cosharbly cocycles $\mu$ and $t$-sharbly 
cycles  $z$.  If $\mu(z)\ne0$, then $z$ represents a nonzero homology class in 
$H_t(G, St)$.

Recall that if $s$ is a simplicial cone in $C^*$ with vertices $v_1',\dots,v_d'$, then $[s]$ denotes the 
sharbly $[v_1,\dots,v_d]$ where the $v_i$ are listed in an order that defines an orientation on $s$ equal to the orientation induced on it by the chosen orientation on $Y$.

\begin{theorem}\label{cocycle}
There exists  a $t$-cosharbly cocycle $\mu$ such that $\mu(z)\ne0$
for any sharbly $z\in Sh_t$ of the form
\[
z=\sum \lambda_i[s_i] +
 \sum \kappa_j  [w^{(j)}_1,\dots,w^{(j)}_{d}],
\] 
where each $\lambda_i>0$, the first sum contains at least one term, and each basic sharbly in the second sum is a flipon.
\end{theorem}

Because  $z_G$ has the property of the chains in the theorem, we obtain:

\begin{corollary}\label{main}
Let $z_G$  be a cycle as constructed in the previous section.
Then $z_G$  represents a nonzero class in $H_t(G, St)$.
\end{corollary}

The proof of Theorem~\ref{cocycle} will occupy the rest of this section.
First some preliminaries.  We have fixed an orientation on $Y$.  The
group $A=\R^\times_+$ acts on $Y$ by homotheties.  We obtain an
induced orientation on $C/A$ by choosing oriented bases of the tangent
space at each point $x$ of $C$ where the last vector points in the
$Ax$ direction (in the direction of increasing $a\in A$.)

The proof uses a theorem from the theory of scissors congruences.
We use some of the notation from
Chapters 1 and 2 of \cite{dupont}, adapted to our situation.  Let $V$
be a real vector space.  A \emph{polytope} $P\subset V$ is a finite union of
simplices $\cup \Delta_i$ such that $ \Delta_i\cap \Delta_j$ is a
common face of lower dimension if $i\ne j$.  
(Note: this is not the usual definition of a polytope, but we are following ~ \cite{dupont}.)

If $P,P_1,P_2$ are
polytopes such that $P=P_1\cup P_2$ and $P_1\cap P_2$ has no interior
points, we write $P=P_1\coprod P_2$.  Let 
$F_V$ be the free
abelian group on the symbols $[P]$ where $P$ runs over all polytopes
in $V$.
Define $\cP_V$ to be $F_V$ modulo the relators 
$[P]-[P_1]-[P_2]$ whenever $P=P_1\coprod
P_2$.

The $A$-action on $C$ commutes with the action of $\SL_n(\R)$.  Then
$C/A$ is the symmetric space for $\SL_n(\R)$ and we fix on it the
orientation described above.  Also, we fix a volume form $\nu$ on
$C/A$ invariant under $\SL_n(\R)$.

\begin{definition}
If $B$ is the cone over a polytope in $C^*$,
then $B/A$ is a polytope in $C^*/A$ and we define
\[
\vol(B)=\int_{B/A} 1\ d\nu,
\]
which is the volume of $B/A$.(\footnote{This is an improper integral, since some of the boundary of $B$
may lie in the boundary of $C^*$.  However, since the total volume of $C/A$ is finite, and the integrand is positive, the integral converges.})
If $[v_1,\dots,v_d]$ is any basic $t$-sharbly, define
\[
\mu([v_1,\dots,v_d])=e\vol (s(v_1,\dots,v_d))
\]
where $e=1$ if the order of the $v_1,\dots,v_d$ is compatible with the orientation on $s(v_1,\dots,v_d)$  induced by the orientation of $Y$, and $e=-1$ otherwise.
\end{definition}

The map  $\mu$ vanishes on the relations (i), (ii) and (iii). If $g\in\SL_n(\Z)$ then $\vol(gP)=\vol(P)$.
Therefore $\mu$ is
a $G$-invariant $\R$-valued cosharbly. It
remains to show that it satisfies the cocycle condition and that
$\mu(z_G)\ne0$.

Let $H$ be the hyperplane in $Y$ defined by $y_{11}=1$, where 
$(y_{ij})$ are the entries in the general symmetric matrix in $Y$.  Then $H\cap C$ is a section for the projection $\pi:C\to C/A$.  We fix the orientation on $H$ such that an oriented frame at a point $x$ of $H$, completed with a vector along $Ax$ in the direction of increasing $a\in A$, is compatible with the orientation we have fixed on $Y$.

For any nonzero vector $v\in\Q^n$, let $v''$ denote $av'$, 
where $a\in A$ is chosen so that $v''\in H$.  Note that $v''$ determines the line $\ell$ through $v$, because kernel of the quadratic form $v'$ is the orthogonal complement to $\ell$ with respect to the standard quadratic form on $\R^n$.

\begin{definition}
For $x_1,\dots,x_m\in H$ define $c(x_1,\dots,x_m)$ to be the convex hull of $x_1,\dots,x_m$. 
\end{definition}

Associate to a basic $t$-sharbly $M=[v_1,\dots,v_d]$ the convex hull $c(v_1'',\dots,v_d'')\subset H$.   The cone on it, namely
$Ac(v_1'',\dots,v_d'')=s(v_1,\dots,v_d)$, is the rational simplicial cone which is the convex hull of
$Av'_1,\dots,Av'_d$.
From the extremal rays $Av'_1,\dots,Av'_d$ we can 
recover $M$ up to sign.   
Knowledge of $c(x_1,\dots,x_d)$, together with
 the orientation determined by the order of
the vertices $x_1'',\dots,x_d''$,   allows us to recover $M$ on the nose, not only up to sign.

The dimension of $H$ is $d-1$. 
\begin{definition} 
We call $c(b_1,\dots,b_{d})$ \emph{proper} if it is $(d-1)$-dimensional. 
\end{definition}

\begin{definition}
\noindent Let $\epsilon(b_1,\dots,b_{d})=1$ if the orientation defined 
on $c(b_1,\dots,b_{d})$ by the ordering $(b_1,\dots,b_{d})$ of its vertices is the same as the chosen orientation on $H$, and $-1$ otherwise.
\end{definition}

We have the following tautology, which we call a ``lemma'' for ease of reference:
\begin{lemma}\label{taut}
$\mu([v_1,\dots,v_d])= \epsilon(v''_1,\dots,v''_d)\vol(c(v''_1,\dots,v''_d))$. 
\end{lemma}

Let $C_*(H)$ be the chain complex where $C_k(H)$ is the free abelian group generated by 
all $k+1$-tuples $(b_1,\dots,b_{k+1})$ where $b_i\in H$.  
We call such a $k+1$-tuple a 
\emph{basic element} of $C_*(H)$.
The boundary map is the usual one:
\[
\partial(b_1,\dots,b_{k+1})=\sum_{i=1}^{k+1} (-1)^{i+1}(b_1,\dots,\widehat b_i,
\dots, b_{k+1}).
\]
The convex 
hull $c(b_1,\dots,b_{k+1})$ of $(b_1,\dots,b_{k+1})$ in $H$ is a (possibly degenerate) $k$-simplex.

For $(b_1,\dots b_{d})\in C_{d-1}$ define
\[
\phi((b_1,\dots b_{d})) = 
\epsilon(b_1,\dots b_{d})c(b_1,\dots b_{d}) 
{\text {\  if  $c$ is proper, and $0$ otherwise}}.
\]

Then Theorem 2.10 of \cite{dupont} says that $\phi$ induces an isomorphism 
\[
\phi: C_{d-1}(H)/(\partial C_d(H)+C_{d-1}(H)^{-})  \to \cP_H,
\]
where $C_{j}(H)^{-}$ denotes the subgroup of $C_{j}(H)$
generated by $(b_1,\dots b_{j+1})$ such that there is an affine subspace of $H$ of 
dimension $j-1$ containing  $\{b_1,\dots b_{j+1}\}$.  

Denote the result of tensoring a $\Z$-module with $\Q$ by a subscript $\Q$.  Since $\Q$ is flat over $\Z$, we obtain
an isomorphism of $\Q$-vector spaces
 \[
\psi: C_{d-1}(H)_\Q/(\partial C_d(H)_\Q
+C_{d-1}(H)^{-}_\Q)  \to (\cP_H)_\Q.
\]

Because basic sharblies of $Sh_*$ are antisymmetric but basic
elements of $C_*(H)$ are not, we define $C^{a}_*(H)$ to be the anti-symmetrized quotient complex of $C_*(H)$:
\[
C^{a}_m(H)=C_m(H)/M_m
\]
where $M_m$ is spanned by elements of the form
\[
(b_1,\dots,b_{m+1})-sign(\tau) (b_{\tau(1)},\dots,b_{\tau(m+1)}),\ \ \tau\in S_{m+1}.
\]
If $(b_1,\dots,b_{m+1})\in C_m(H)$, we denote its image in $C^{a}_m(H)$
by $(b_1,\dots,b_{m+1})^a$.

The boundary map descends to $C^a_*(H)$ and we define $C^a_{d-1}(H)^{-}$ to be the image of $C_{d-1}(H)^{-}$ in $C^a_{d-1}(H)$.  Since $\phi$ vanishes on $M_{d-1}$,
 $\psi$ descends to a map
 \[
\psi: C^a_{d-1}(H)_\Q/(\partial C^a_d(H)_\Q
+C^a_{d-1}(H)^{-}_\Q)  \to (\cP_H)_\Q.
\]

\begin{definition}
For $m\ge0$, define the linear map of $\Q$-vector spaces 
\[f_m:Sh_m\to C^a_{m+n-1}(H)_\Q/C^a_{m+n-1}(H)^{-}_\Q\] 
by setting
$f_m([v_1,\dots,v_{m+n}])=(v''_1,\dots,v''_{m+n})^a$ on  basic sharblies and extending by linearity. 
\end{definition}
Clearly $f_m$ commutes with the boundary maps, so it induces a map
\[f_m:Sh_m/\partial(Sh_{m+1})\to C^a_{m+n-1}(H)_\Q/
(\partial C^a_{m+n}+C^a_{m+n-1}(H)^{-}_\Q).\]

We now complete the proof of the theorem. 
The composition
\[\psi\circ f_t: Sh_t/\partial(Sh_{t+1}) \to (\cP_H)_\Q\] maps
a basic sharbly 
$[v_1,\dots,v_{d}]$ to the class of  
$\epsilon(v''_1,\dots,v''_{d})c(v''_1,\dots,v''_{d})$.

\begin{definition}
For any polytope $P$ in $H$, define $\vol(P)=\vol(AP)$.
\end{definition}
We extend $\vol$ to a function 
$\widehat\vol:F_V\to\R$ by linearity.
In particular $\widehat\vol(-[P])=-\widehat\vol([P])$.
If $P_1,P_2$ are two polytopes in $H$, then 
$\vol(P_1\coprod P_2) = \vol(P_1) +  \vol(P_2)$.  It follows that  $\widehat\vol$ descends to a homomorphism of abelian groups
$\widehat\vol:(\cP_H)_\Q\to\R$.

By Lemma~\ref{taut}, $\mu= \widehat\vol\circ\psi\circ f_t$.  Therefore
$\mu$ vanishes on boundaries and hence is a $t$-cosharbly cocycle for
$G$.  Also, $\mu$ vanishes on basic sharblies $[w_1,\dots,w_t]$ with
the property that $w'_1,\dots,w'_t$ lie in a hyperplane in $Y$,
because $\psi$ vanishes on $C_{d-1}(H)_\Q^-$.  Hence $\mu(z)$ is the
sum of positive terms, and thus is nonzero.  This completes the proof
of Theorem~\ref{cocycle}

\section{Examples}
In this section we discuss the cases where $n=2,3,4,5$ and $G=\SL_n(\Z)$.  
Let $\{e_i\}$ denote the standard basis of $\Q^n$.

For $n=2$ there is one $\SL_2(\Z)$-orbit of tiles, represented by a
tile with vertices corresponding to the minimal vectors
$e_1,e_2,e_1-e_2$.  The corresponding perfect quadratic form is known
as $A_{2}$.(\footnote{More information about perfect forms can be
found at the sources \cite{cs,cs2,catalogue,MM}.})  The boundary
of this tile is a simplex, so there are no flipon terms in $z_G$.  The
stabilizer of this tile in $\SL_2(\Z)$ has order 6.  Then
\[
z_G=(1/6)[e_1,e_2,e_1-e_2]_G.
\]
One can see directly that $\partial z_G=0$.  In fact,
\[
\partial z_G=[e_2,e_1-e_2]_G-[e_1,e_1-e_2]_G+[e_1,e_2]_G.
\]
Let 
\[
g=\begin{bmatrix} 0&1\\-1&0\end{bmatrix},\ 
h=\begin{bmatrix} 0&1\\-1&-1\end{bmatrix}
\in\SL_2(\Z).
\]
Then $g[e_1,e_2]=[e_2,e_1]=-[e_1,e_2]$, so the last term on the right
hand side equals 0.  Also $h[e_1,e_2]=[e_2,e_1-e_2]$, so $hgh\inv
[e_2,e_1-e_2]=- [e_2,e_1-e_2]$ and the first term on the right hand
side equals 0.  Similarly the middle term equals 0.
 
For $n=3$ there is only one $\SL_3(\Z)$-orbit of tiles, represented by
the tile $T_0$ with vertices corresponding to the minimal vectors
$e_1,e_2,e_3,e_1-e_2,e_1-e_3,e_2-e_3$; the corresponding perfect form
is known as $A_{3}$.  The boundary of this tile is a simplex, so there
are no flipons in $z_G$.  The stabilizer of this tile has order 24.
Then
\[
z_G=(1/24)[e_1,e_2,e_3,e_1-e_2,e_1-e_3,e_2-e_3]_G.
\]
One can also see easily in this case that $\partial z_G=0$. There are
six terms in the boundary of $z_G$.  The stabilizer of $T_0$ acts
transitively on them, so it suffices to show that any one of them
equals zero, for example, to show that
\[
[e_1,e_2,e_3,e_1-e_2,e_2-e_3]_G=0.
\]
Let 
\[
k=\begin{bmatrix} 0&0&-1\\0&1&1\\1&0&0\end{bmatrix}
\in\SL_3(\Z).
\]
Then $$k[e_1,e_2,e_3,e_1-e_2,e_2-e_3]=[e_3,e_2,e_1-e_2,e_2-e_3,e_1]=
-[e_1,e_2,e_3,e_1-e_2,e_2-e_3].$$

For $n=4$
there are two $\SL_4(\Z)$-orbits of tiles, represented by $T_0$ and
$T_1$.  Here $T_0$ is a simplicial cone, corresponding to the perfect
form $A_{4}$, so it gives one term in
$z_G$.  The tile $T_1$ corresponds to the perfect form known as
$D_{4}$; it is a polytope with 12 vertices, and can be
subdivided into 16 simplicial cones, giving 16 terms in $z_G$.  The
facets of $T_{1}$ are all simplices, so we don't need any flipons.
Still, we would not want to check $\partial z_G=0$ by hand.

If the reader should wish to write down $z_G$ explicitly for $n=4$,
this can be done using the following information.  The tile $T_0$ has
vertices corresponding to the vectors $e_1,...,e_4,e_i-e_j$ (for $1\le
i < j \le 4$).  The vertices of the 16 simplicial cones whose union
are the tile $T_1$ correspond to sets of vectors as follows:

Let $[a,b,c,d]$ denote the vector $ae_1+be_2+ce_3+de_4$.  The vertices
of $T_1$ correspond to the column vectors in the following matrix:

\[
\begin{bmatrix}
-1&-1&-1&-1&-1&-1&0&0&0&0&0&0\\
-1&0&0&0&0&1&-1&-1&-1&-1&0&0\\
0&-1&0&0&1&0&-1&0&0&1&-1&-1\\
1&1&0&1&0&0&1&0&1&0&0&1
\end{bmatrix}
\]

Label the columns $0$, $1$, \dots , $11$.  Then we can describe each
of the simplicial cones whose union make up $T_1$ by giving its
vertices, according to these labels.  Of course, such a simplicial
subdivision is not unique, but one we have found has 16 simplicial
cones.  Here we give the lists of vertices for each cone:

\begin{center}
\begin{tabular}{lll}
\{0, 1, 2, 3, 4, 5, 6, 7, 8, 10\},&\{0, 1, 3, 4, 5, 6, 7, 8, 9, 10\},&\{0, 1, 2, 4, 5, 6, 7, 8, 9, 10\},\\
\{0, 1, 2, 3, 4, 5, 7, 8, 9, 10\},&\{1, 2, 3, 4, 5, 6, 7, 8, 10, 11\},&\{1, 3, 4, 5, 6, 7, 8, 9, 10, 11\},\\
\{1, 2, 4, 5, 6, 7, 8, 9, 10, 11\},&\{1, 2, 3, 4, 5, 7, 8, 9, 10, 11\},&\{0, 1, 2, 3, 4, 6, 7, 8, 10, 11\},\\
\{0, 1, 3, 4, 6, 7, 8, 9, 10, 11\},&\{0, 1, 2, 4, 6, 7, 8, 9, 10, 11\},&\{0, 1, 2, 3, 4, 7, 8, 9, 10, 11\},\\
\{0, 1, 2, 3, 4, 5, 7, 8, 9, 11\},&\{0, 1, 2, 4, 5, 6, 7, 8, 9, 11\},&\{0, 1, 3, 4, 5, 6, 7, 8, 9, 11\},\\
&\{0, 1, 2, 3, 4, 5, 6, 7, 8, 11\}.&\\
\end{tabular}
\end{center}

For $n=5$, there are three tiles modulo $\SL_5(\Z)$, $T_0$, $T_1$, and
$T_2$.  The vertices of $T_0$, which corresponds to the perfect form
$A_{5}$, correspond to $e_1,...,e_5,e_i-e_j$ (for $1\le i < j \le 5$).
This is a simplicial cone, so naturally all of its facets are
simplicial.  The tile $T_{1}$ corresponds to the perfect form known as
$A_{5}^{+3}$; its vertices correspond to the columns of
\[
\begin{bmatrix}
 1&0&1&0&0&0&1&0&1&0&0&0&1&0&1\\
 1&1&0&0&0&0&1&1&0&0&1&1&0&0&0\\
 1&1&1&1&1&1&0&0&0&0&0&0&0&0&0\\
 1&1&1&0&1&0&1&0&0&0&1&0&1&1&0\\
 1&1&1&1&0&0&1&1&1&1&0&0&0&0&0
\end{bmatrix}
\]

The cone $T_1$ is also simplicial, so all of its facets are
simplicial.  So for these two cones no flipons are needed.  Finally
the cone $T_2$ for the perfect form $D_{5}$ has 20 spanning rays,
corresponding to the vectors
\[
\begin{bmatrix}
 0&1&0&0&0&0&1&0&1&0&1&0&1&0&0&1&0&0&1&1\\
 0&0&1&0&0&1&0&0&0&1&1&0&0&1&0&0&1&1&-1&0\\
 0&0&0&1&1&0&0&0&0&0&0&1&1&1&1&-1&-1&0&0&0\\
 1&-1&-1&-1&1&1&1&1&0&0&0&0&0&0&0&0&0&0&0&0\\
 0&0&0&0&-1&-1&-1&-1&-1&-1&-1&-1&-1&-1&0&0&0&0&0&0
\end{bmatrix}
\]
It has 400 facets, 320 of which are simplicial (and therefore have 14
spanning rays), but 80 of which have 16 spanning rays.  We will not
attempt to write down $z_G$ in this case, but it may be of interest to
give an example of two non-simplicial facets of $T_2$ that could be
mapped one to the other by an element of $SL_5(\Z)$ together with the
flipons necessary to convert one simplicial subdivision into the
other.  The following data was obtained by the use of the programs
\texttt{Sage} \cite{sage} and \texttt{polymake} \cite{polymake}.

Let us index the vertices of $T_2$ by $0,\dots,19$.
Then one of its non-simplicial facets $F$ has the 16 vertices
$[0,1,3,4,5,6,7,8,9,11,12,13,14,15,18,19]$.
We can triangulate $F$ with the following simplices:
\begin{gather*}
\{0,1,3,4,5,6,7,8,9,11,12,13,14,18\}\\
\{0,1,3,4,5,6,7,8,9,11,13,14,15,18\}\\
\{0,1,3,4,5,6,7,9,11,12,13,14,18,19\}\\
\{0,1,3,4,5,6,7,9,11,13,14,15,18,19\}\\
\{0,1,3,4,5,6,8,9,11,12,13,14,15,18\}\\
\{0,1,3,4,5,6,9,11,12,13,14,15,18,19\}\\
\{0,1,3,5,6,7,8,9,11,12,13,14,18,19\}\\
\{0,1,3,5,6,7,8,9,11,13,14,15,18,19\}\\
\{0,1,3,5,6,8,9,11,12,13,14,15,18,19\}\\
\{0,1,4,5,6,7,8,9,11,12,13,14,18,19\}\\
\{0,1,4,5,6,7,8,9,11,13,14,15,18,19\}\\
\{0,1,4,5,6,8,9,11,12,13,14,15,18,19\}\\
\{1,3,4,5,6,7,8,9,11,12,13,14,15,18\}\\
\{1,3,4,5,6,7,9,11,12,13,14,15,18,19\}\\
\{1,3,5,6,7,8,9,11,12,13,14,15,18,19\}\\
\{1,4,5,6,7,8,9,11,12,13,14,15,18,19\}.
\end{gather*}
We can imagine a different triangulation, perhaps coming from another facet in the $G$-orbit of $F$, for example
\begin{gather*}
\{0,1,3,4,5,6,7,8,9,11,12,13,15,18\}\\
\{0,1,3,4,5,6,7,8,9,12,13,14,15,18\}\\
\{0,1,3,4,5,6,7,9,11,12,13,15,18,19\}\\
\{0,1,3,4,5,6,7,9,12,13,14,15,18,19\}\\
\{0,1,3,4,5,7,8,9,11,12,13,14,15,18\}\\
\{0,1,3,4,5,7,9,11,12,13,14,15,18,19\}\\
\{0,1,3,5,6,7,8,9,11,12,13,15,18,19\}\\
\{0,1,3,5,6,7,8,9,12,13,14,15,18,19\}\\
\{0,1,3,5,7,8,9,11,12,13,14,15,18,19\}\\
\{0,1,4,5,6,7,8,9,11,12,13,15,18,19\}\\
\{0,1,4,5,6,7,8,9,12,13,14,15,18,19\}\\
\{0,1,4,5,7,8,9,11,12,13,14,15,18,19\}\\
\{0,3,4,5,6,7,8,9,11,12,13,14,15,18\}\\
\{0,3,4,5,6,7,9,11,12,13,14,15,18,19\}\\
\{0,3,5,6,7,8,9,11,12,13,14,15,18,19\}\\
\{0,4,5,6,7,8,9,11,12,13,14,15,18,19\}.
\end{gather*}

In fact, the secondary polytope of $F$ has 3 vertices, so these are 2
of the 3 regular triangulations of this facet corresponding to the
vertices of secondary polytope.  The flip that takes the first of
these triangulations to the second is described by a pair $(T_+,T_1)$ contained in a
circuit.  In this flip, the (cones on the) simplices in $T_+$ are
replaced by (the cones on) those in $T_-$.  To describe the flip, we
renumber the vertices of $F$ as $0\dots,15$, because this was mandated
by the software we use.

The simplices in $T_+$ and $T_-$ are all 6-simplices.  $T_+$  is a list
of 4 such, and $T_-$ is a list of 4 such.  
\begin{gather*}
T_+=\{\{0,1,5,6,9,10,13\}, \{0,1,5,6,10,12,13\}, \\ \{0,1,6,9,10,12,13\}, \{0,5,6,9,10,12,13\}\},\\
T_-=  \{\{0,1,5,6,9,10,12\}, \{0,1,5,6,9,12,13\},\\ \{0,1,5,9,10,12,13\}, \{1,5,6,9,10,12,13\}\}.
 \end{gather*}
  
We list only the vertices that change in the flip, the rest staying
the same.  In other words, this flip has non-maximal dimension.  The
circuit consists of 8 vertices, and the simplices of $T_+$ and $T_-$
are coned off by the remaining vertices to achieve the relevant
triangulations of $F$.

The situation is the same with all three possible regular
triangulations of the facet $F$.  Any two of them are connected by
just one flipon, where the circuit has 8 vertices.

\begin{remark} One might hope, generalizing from the cases $n=2,3$,
that we could construct a sharbly cycle in the following way: Let
$A_n$ be the the perfect form whose minimal vectors are
$\{e_1,\dots,e_n\} \cup \{e_i-e_j\ | \ 1\le i < j \le n \}$.
Enumerate the minimal vectors as $v_1,\dots, v_d$.  Then perhaps
\[
\partial [v_1,\dots,v_d]_G=0.
\]
Unfortunately, this is only true for $n=2$ or $3$.  Here is the reason:

First of all, $ [v_1,\dots,v_d]=[s_{A_n}]$.  Theorem 7.5.1
in~\cite{MM} implies that the tile $s_{A_n}$ has a unique $G$-orbit of
facets, and if $n>3$, if $T$ is a tile such that $T\cap s_{A_n}$ is a
facet of both $T$ and $s_{A_n}$, then $T$ is not in the same $G$-orbit
as $s_{A_n}$.  It follows that if $n>3$ then the $G$-stabilizer of a
facet $F$ of $s_{A_n}$ is a subgroup of the $G$-stabilizer of
$s_{A_n}$.  Since $G$ preserves orientation on $C$, there will be no
element in the $G$-stabilizer of $F$ that reverses the orientation of
$F$, and therefore $\partial
[v_1,\dots,v_d]_G=d[v_2,\dots,v_d]_G\ne0$.\end{remark}

\section{Appendix: Triangulations of Polytopes}\label{app}

In this appendix we gather the results we need about triangulations of
polytopes.  If the polytopes in question are all contained in some
hyperplane of $Y$ that cuts $C^*-\{0\}$ transversally, then all these
results immediately carry over to the cones on the polytopes with
vertex $0$.  The reference for all these theorems is chapter 7
in~\cite{GKZ}.  In this section we use the definition of ``polytope''
found in~\cite{GKZ}, rather than the one we used before when
discussing scissors congruences.  Every polytope in the new sense is
also a polytope in the old sense.  Let $V$ denote an affine space over
$\R$ of dimension $n$.  We fix an orientation on $V$.

\begin{definition}
A \emph{polytope} $P$ is the convex hull of a finite number of points
of $V$.   
If $P$ is of full dimension $n$, we give it the orientation
induced from $V$.  The vertices of $P$ are its extreme points
and $\ver(P)$
denotes the set of vertices of $P$.   A
\emph{facet} of $P$ is a codimension 1 face of it. 

A \emph{triangulation} of $P$ is a
collection of simplices $\{\sigma_i\}$, such that for each $i$,
$\ver(\sigma_i)\subset\ver(P)$, $\cup_i\sigma_i=P$ and if $i\ne j$,
$\sigma_i\cap\sigma_j$ is either empty or a common face of $\sigma_i$
and $\sigma_j$.  A \emph{regular triangulation} of $P$ is one that
obeys Definition 1.3 in Chapter 7 of~\cite{GKZ}.(\footnote{These are
called \emph{coherent triangulations} in~\cite{GKZ}, but most later
authors call them \emph{regular triangulations}.})  A different but equivalent
definition may be found in Section 16.3 of
\cite{book}.(\footnote{Chapter 16 by C. W. Lee and F. Santos in this
book collects a lot of useful information about polytopes and their
triangulations.})

If $S$ is any subset of $V$ let $c(S)$ denote the convex hull of $S$.
\end{definition}

If $g$ is an affine transformation of $V$ and $\{\sigma_i\}$ is a triangulation of $P$, then $\{g \sigma_i\}$ is  a triangulation of $P$.  
If $\{\sigma_i\}$ is regular, then $\{g \sigma_i\}$ is also regular.  If $Q$ is a facet of $P$, and $\{\sigma_i\}$ is a regular triangulation of $P$, then the triangulation of $Q$ induced by $\{\sigma_i\}$ is regular.

If $v_1,\dots,v_{m}\in V$, let $c(v_1,\dots,v_{m})$ denote their convex hull.     If
 $\ver(P)=\{v_1,\dots,v_m\}$, then $P=c(v_1,\dots,v_m)$.  
 If $m=n+1$ and $v_1,\dots,v_{m}$ are affinely independent, then
$c(v_1,\dots,v_{m})$ is a simplex.
We give it the orientation induced by $V$. 
 
\begin{definition}  Let $m\le n+2$.
A \emph{circuit} is a set of $m$ points in $V$ such that any $m-1$ of them are affinely independent.  
\end{definition}
If $m=n+2$ we say that the circuit is ``of maximal dimension''.
Note that if $Z$ is a circuit, then every facet of $c(Z)$ is a simplex.

Proposition 1.2 of~\cite{GKZ} says the following:  Let $Z$ be a circuit.
Then $c(Z)$ has exactly two triangulations, $T_+$ and $T_-$.  They can be described as follows:   There is a partition of $Z$ into two nonempty sets: $Z=Z_+\coprod Z_-$. such that $T_+$ consists of the simplices $c(Z-\omega)$, 
as $\omega$ ranges through the elements of $Z_+$, and $T_-$ consists of the simplices $c(Z-\omega)$
for $\omega\in Z_-$.

\begin{definition}
Let $P$ be a polytope and $Z$ a circuit contained in $\ver(P)$.  Suppose $P$ is triangulated, so that $c(Z)\subset P$ has the induced triangulation $T_+$.  Then if we instead triangulate $c(Z)$ by $T_-$, and take the triangulation of $P$ given by all simplices of the form
$c(I\cup F)$ where $I$ is a simplex (possibly empty) in $T_-$ and $F$ is any subset (possibly empty) of 
$\ver(P)-Z$, we 
say that the new triangulation of $P$ is obtained from the old one by a \emph{flip}.(\footnote{This is called a \emph{modification} in~\cite{GKZ}.})
Similarly if we reverse the roles of $T_+$ and $T_-$, we say that new triangulation is obtained from the old one by a flip. 
\end{definition}

The \emph{secondary polytope} $\Sigma(P)$ is defined in Definition 1.6 of Chapter 7 of~\cite{GKZ}.  Being a polytope, it is nonempty and connected.  By Theorem 1.7 of the same source, the vertices of $\Sigma(P)$ correspond one-to-one with all regular triangulations of $P$.
By Theorem 2.10 of the same source, the edges of $\Sigma(P)$ correspond to flips that take one vertex of an edge to the other vertex of that edge.  We conclude that any two regular triangulations of $P$ are connected by a sequence of flips.

We want to write a useful formula for a flip.  Let $p=m+2$.
 Let $T_1$ and $T_2$ be regular triangulations of $P$ connected by a flip with respect to the circuit $Z=\{z_1,\dots,z_{p}\}$, where 
$\{z_1,\dots,z_{p}\}\subset \ver(P)$.  
Let $s_1,\dots,s_r$ be the simplices in $T_1$ listed so that 
 $s_1,\dots,s_k$ are the simplices that get removed in the flip,
 to be replaced by the simplices $u_1,\dots,u_{p-k}$.
 Then the simplices in $T_2$ are
 $u_1,\dots,u_{p-k},s_{k+1},\dots,s_r$.  
 We give all these simplices the orientation induced from $V$.
Then
 \[
\bigcup_i \{c(z_1,\dots, \widehat z_i,\dots, z_{p})\} =
 \{s_1,\dots,s_k,u_1,\dots,u_{p-k}\}.
 \]

  Let $A_{r}$ be the free $\Z$-module on symbols $(a_1,\dots,a_{r})$ with $a_i \in V$, modulo the relations 
\[
(a_{\sigma(1)},\dots,a_{\sigma({r})})=\sign(\sigma) (a_1,\dots,a_{r})
\]
for all $\sigma\in S_{r}$.
If $a$ is an oriented $r-1$-simplex, write $a^*$ for the symbol 
$(a_1,\dots,a_{r})$ where the $a_i$ run through the vertices of $a$ 
and are in an order that induces the given orientation on $a$.

\begin{theorem}\label{glue}
Notations as above.  Then in $A_{p-1}$ we have the equality
 \[
e\sum_i (-1)^i (z_1,\dots, \widehat z_i,\dots, z_{p}) =
(s_1^*+\dots +s_k^*) - (u_1^*+\dots+u^*_{p-k}),
 \]
 where $e$ is either 1 or $-1$.
 \end{theorem}
 
 \begin{proof}
 Define the boundary map $\partial:A_p\to A_{p-1}$ by 
 \[
 (a_1,\dots,a_p)=\sum_i (-1)^i (a_1,\dots,\widehat a_i,\dots,a_p).
 \]
We know that the right hand side of the displayed equation in the statement of the theorem equals
$\sum_i \epsilon(i) (z_1,\dots, \widehat z_i,\dots, z_{p})$ for some choice of 
signs $\epsilon(i)=\pm1$.  Because it is the difference of two triangulations of the same polytope $c(Z)$, its boundary is 0.  So we will be finished when we prove the following lemma.

 \begin{lemma}\label{signs}
 The only choices  of signs $\epsilon(i)=\pm1$ such that
 \[
 \partial \sum_i \epsilon(i) (z_1,\dots, \widehat z_i,\dots, z_{p})=0
 \]
 are $\epsilon(i)=(-1)^i$ or $\epsilon(i)=(-1)^{i+1}$.
 \end{lemma}

Compute:  
\begin{gather*}
\partial \sum_i \epsilon(i) (z_1,\dots, \widehat z_i,\dots, z_{p})=\\
\sum_{j<i}(-1)^j  \sum_i \epsilon(i) (z_1,\dots, \widehat z_j,\dots,\widehat z_i,\dots, z_{p})+\\
\sum_{j>i}(-1)^{j-1}  \sum_i \epsilon(i) (z_1,\dots, \widehat z_i,\dots,\widehat z_j,\dots, z_{p}).
\end{gather*}
The term $(z_1,\dots,z_p)$ with both $z_i$ and $z_j$ omitted occurs twice and the sum of those terms must be 0.
In particular, choose $a>1$ and take $j=1$ and $i=a$ in the first sum, and $i=1$, $j=a$ in the second.  Then
$-\epsilon(a) (z_2,\dots,\widehat z_a,\dots, z_{p})+
(-1)^{a-1}   \epsilon(1) (z_2,\dots,\widehat z_a,\dots, z_{p})=0$.
It follows that $\epsilon(a)=(-1)^{a-1} \epsilon(1)$ for all $a$.
\end{proof}

\begin{corollary}\label{glue0}
Let $P$ be a polytope of dimension $n$ and $p=m+2\le n+1$.
 Let $T_1$ and $T_2$ be regular triangulations of $P$ connected by a flip with respect to the circuit $Z=\{z_1,\dots,z_{p}\}$, where 
$\{z_1,\dots,z_{p}\}\subset \ver(P)$.  Let $x_{p+1},\dots,x_{n+2}$ be the remaining vertices of $P$.
Let 
 $w_1,\dots,w_k$ be the simplices that get removed in the flip,
 to be replaced by the simplices $y_1,\dots,y_{p-k}$.
Then
 \[
\sum_{i=1}^{p} 
(-1)^i (z_1,\dots, \widehat z_i,\dots, z_{p},x_{p+1},\dots,x_{n+2}) =
(w_1^*+\dots +w_k^*) - (y_1^*+\dots+y^*_{p-k})
 \]
\end{corollary}

\begin{remark}
An example might help.  Suppose $n=3$, $p=4$ and the circuit $1,2,3,4$ consisting of vertices that form a square in a plane.   Let $P$ be the pyramid which is the cone from vertex $5$ on that square $[1234]$.  Consider the flip (not of maximal dimension) that takes the triangulation $\{[1235],[1345]\}$ to the triangulation $\{[1245],[2345]\}$.  Then the formula in Corollary~\ref{glue0} becomes
\[
-[2345]+[1345]-[1245]+[1235] = ([1235]+[1345]) - ([2345]+[1245]).
\]
\end{remark}

\begin{corollary}\label{glue1}
Let $P$ be a polytope of full dimension $n$.
Let $\{w_1,\dots,w_m\}$ and $\{y_1,\dots,y_{b}\}$ be two regular triangulations of $P$, where all the orientations of the maximal simplices are those induced by $V$.

  Then there exists 
  a positive integer $k$, and for  $1\le\alpha\le k$
non-repeating sequences  $z^\alpha_1,\dots, z^\alpha_{p(\alpha)}$ and $x^\alpha_{p(\alpha)+1},\dots,x^\alpha_{n+2}$,
all consisting of vertices of $P$,
  such that in $A_{n+2}$ we have the equality
 \[
\sum_\alpha \sum_{i=1}^{p(\alpha)} 
(-1)^i (z^\alpha_1,\dots, \widehat z^\alpha_i,\dots, z^\alpha_{p(\alpha)},x^\alpha_{p(\alpha)+1},\dots,x^\alpha_{n+2}) =
(w_1^*+\dots +w_m^*) - (y_1^*+\dots+y^*_{b}),
 \]
 where
$ \{x^\alpha_{p(\alpha)+1},\dots,x^\alpha_{n+2}\}=
\ver(P)-\{z^\alpha_1,\dots, z^\alpha_{p(\alpha)}\}$.
 \end{corollary}
 
 \begin{proof}
 This follows from the theorem because any two regular triangulations are connected by a sequence of flips.   
 \end{proof}

We now apply this to triangulations of tiles in $C^*$.
Let $H$ be the hyperplane of $Y$ defined in Section~\ref{cococycle} and as in that section if $v\in\Q^n$, let $v''$ be the intersection of the ray through $v'$ with $H$.  If $v_1,\dots,v_m\in\Q^n$, let $c(v''_1,\dots,v''_m)$ be the convex hull in $H$.  Then the convex cone $s(v'_1,\dots,v'_m)$ generated by $v'_1,\dots,v'_m$ is the cone over $c(v''_1,\dots,v''_m)$, and the intersection of $s(v_1,\dots,v_m)$ with $H$ is $c(v''_1,\dots,v''_m)$.  
When $m=d$ and $v'_1,\dots,v'_d$ are linearly independent in $Y$, we
 give the simplex $c(v''_1,\dots,v''_d)$ the orientation determined by the order 
 $v_1,\dots,v_d$.
 
Suppose $s$ is a simplicial cone in $C^*$ with vertices $v_1',\dots,v_d'$.  As usual, we give $s=s(v_1,\dots,v_d)$  
the orientation  induced by the fixed orientation on $Y$.
As in the last bullet of Definition~\ref{def1}, we list the $v_i$ in an order that defines this orientation on 
$s$ and define the
sharbly $[s]=[v_1,\dots,v_d]$, where the $v_i$ are listed in this order.  Then $c(v''_1,\dots,v''_d)$ also has the orientation induced from $Y$.
 
 \begin{definition} Let $X$ be a tile or a facet.  Then a
\emph{regular triangulation} of $X$ is a triangulation of it whose
intersection with $H$ is a regular triangulation of $X\cap
H$.  \end{definition}
 
 It follows from what was said earlier in this section that if $g\in\SL_n(\Z)$ and $\{s\}$ is a regular
  triangulation of $X$ then $\{gs\}$ is a regular
  triangulation of $gX$.  If  $\{s\}$ is a regular
  triangulation of a tile $T$, then the
  triangulation it induces on a facet of $T$ is regular.
  
Recall Definition~\ref{flip}: A flipon is a basic $t$-sharbly
$[v_1,\dots,v_d]$ such that there is an affine $(d-2)$-space in $Y$
that contains $v_i'$ for all $i=1,\dots,d$.

\begin{corollary}\label{glue2}
Let there be two regular compatibly 
oriented triangulations $\{w\}$ and $\{y\}$  of the same oriented facet of a tile.  Then there exists flipons $[v^\alpha_1,\dots,v^\alpha_d]$ and integers $p(\alpha)\ge3$ such that 
 \[
\sum_\alpha \sum_{i=1}^{p(\alpha)} (-1)^i
[v^{\alpha}_1,\dots, \widehat {v^{\alpha}_i},\dots, v^{\alpha}_{d}] =
\sum [w] - \sum [y].
 \]
\end{corollary}
Note that the alternating sum over $i$ is only part of the boundary of the flipon-sharbly, except when the flip has maximal dimension. 
\begin{proof}
This follows from easily Corollary~\ref{glue1}.   The reason $p(\alpha)>2$ is that a circuit has to have at least three elements.  \end{proof}

\bibliographystyle{amsalpha}
\bibliography{AGM-Puzzle-Paper}

\newcommand{\etalchar}[1]{$^{#1}$}
\providecommand{\bysame}{\leavevmode\hbox to3em{\hrulefill}\thinspace}
\providecommand{\MR}{\relax\ifhmode\unskip\space\fi MR }
\providecommand{\MRhref}[2]{%
  \href{http://www.ams.org/mathscinet-getitem?mr=#1}{#2}
}
\providecommand{\href}[2]{#2}
\begin{thebibliography}{SEVKM19}

\bibitem[AGM12]{AGM5}
Avner Ash, Paul~E. Gunnells, and Mark McConnell, \emph{Resolutions of the
  {S}teinberg module for {$GL(n)$}}, J. Algebra \textbf{349} (2012), 380--390.
  \MR{2853645}

\bibitem[AMRT10]{AMRT}
Avner Ash, David Mumford, Michael Rapoport, and Yung-Sheng Tai, \emph{Smooth
  compactifications of locally symmetric varieties}, second ed., Cambridge
  Mathematical Library, Cambridge University Press, Cambridge, 2010, With the
  collaboration of Peter Scholze. \MR{2590897 (2010m:14067)}

\bibitem[Ash94]{unstable}
Avner Ash, \emph{Unstable cohomology of {${\rm SL}(n,\mathcal O)$}}, J. Algebra
  \textbf{167} (1994), no.~2, 330--342. \MR{MR1283290 (95g:20050)}

\bibitem[Ash24]{AA}
Avner Ash, \emph{On the cohomology of $\mathrm{SL}_n(\mathbb{Z})$}, 2024.

\bibitem[BCG23]{bcg}
George Boxer, Frank Calegari, and Toby Gee, \emph{{Cuspidal cohomology classes
  for $\mathrm{GL}_n({\bf Z})$}}, 2023.

\bibitem[BMP{\etalchar{+}}22]{BMPSW}
Benjamin Br{\"u}ck, Jeremy Miller, Peter Patzt, Robin~J. Sroka, and Jennifer
  C.~H. Wilson, \emph{On the codimension-two cohomology of
  $\mathrm{SL}_n(\mathbb{Z})$}, 2022.

\bibitem[Bro23]{FB}
Francis Brown, \emph{Bordifications of the moduli spaces of tropical curves and
  abelian varieties, and unstable cohomology of $\mathrm{GL}_g(\mathbb{Z})$ and
  $\mathrm{SL}_g(\mathbb{Z})$}, 2023.

\bibitem[BS73]{B-S}
Armand Borel and Jean-Pierre Serre, \emph{Corners and arithmetic groups}, Comm.
  Math. Helv. \textbf{48} (1973), 436--491.

\bibitem[CFP14]{CFP}
Thomas Church, Benson Farb, and Andrew Putman, \emph{A stability conjecture for
  the unstable cohomology of {${\rm SL}_n(\Bbb Z)$}, mapping class groups, and
  {${\rm Aut}(F_n)$}}, Algebraic topology: applications and new directions,
  Contemp. Math., vol. 620, Amer. Math. Soc., Providence, RI, 2014, pp.~55--70.
  \MR{3290086}

\bibitem[CS88]{cs}
J.~H. Conway and N.~J.~A. Sloane, \emph{Low-dimensional lattices. {III}.
  {P}erfect forms}, Proc. Roy. Soc. London Ser. A \textbf{418} (1988),
  no.~1854, 43--80. \MR{953277}

\bibitem[CS89]{cs2}
\bysame, \emph{Errata: ``{L}ow-dimensional lattices. {III}. {P}erfect forms''
  [{P}roc. {R}oy. {S}oc. {L}ondon {S}er. {A} {\bf 418} (1988), no. 1854,
  43--80; {MR}0953277 (90a:11073)]}, Proc. Roy. Soc. London Ser. A \textbf{426}
  (1989), no.~1871, 441. \MR{1030469}

\bibitem[Dup01]{dupont}
Johan~L. Dupont, \emph{Scissors congruences, group homology and characteristic
  classes}, Nankai Tracts in Mathematics, vol.~1, World Scientific Publishing
  Co., Inc., River Edge, NJ, 2001. \MR{1832859}

\bibitem[EVGS13]{EVGS}
Philippe Elbaz-Vincent, Herbert Gangl, and Christophe Soul\'{e}, \emph{Perfect
  forms, {K}-theory and the cohomology of modular groups}, Adv. Math.
  \textbf{245} (2013), 587--624. \MR{3084439}

\bibitem[GG]{GG}
N.~Grbac and H.~Grobner, \emph{On the cohomology of $\mathrm{SL}_n(\mathbb{Z})$
  beyond the stable range}, preprint.

\bibitem[GJ00]{polymake}
Ewgenij Gawrilow and Michael Joswig, \emph{polymake: a framework for analyzing
  convex polytopes}, Polytopes---combinatorics and computation ({O}berwolfach,
  1997), DMV Sem., vol.~29, Birkh\"auser, Basel, 2000, pp.~43--73. \MR{1785292}

\bibitem[GKT21]{GKT}
Jeffrey Giansiracusa, Alexander Kupers, and Bena Tshishiku,
  \emph{Characteristic classes of bundles of {K}3 manifolds and the {N}ielsen
  realization problem}, Tunis. J. Math. \textbf{3} (2021), no.~1, 75--92.
  \MR{4103767}

\bibitem[GKZ94]{GKZ}
I.~M. Gelfand, M.~M. Kapranov, and A.~V. Zelevinsky, \emph{Discriminants,
  resultants, and multidimensional determinants}, Modern Birkh{\"a}user
  Classics, Springer-Verlag, 1994.

\bibitem[GOT17]{book}
J.E. Goodman, J.~O'Rourke, and C.~D. T{\'o}th (eds.), \emph{Handbook of
  discrete and computational geometry}, third ed., CRC Press, Boca Raton, FL,
  2017.

\bibitem[LS76]{LS}
Ronnie Lee and R.~H. Szczarba, \emph{On the homology and cohomology of
  congruence subgroups}, Invent. Math. \textbf{33} (1976), no.~1, 15--53.
  \MR{0422498 (54 \#10485)}

\bibitem[LS78]{LS2}
\bysame, \emph{On the torsion in {$K\sb{4}({\bf Z})$} and {$K\sb{5}({\bf
  Z})$}}, Duke Math. J. \textbf{45} (1978), no.~1, 101--129. \MR{491893}

\bibitem[Mar03]{MM}
J.~Martinet, \emph{Perfect lattices in euclidean spaces}, Grundlehren der
  mathematischen Wissenschaften, vol. 327, Springer-Verlag, 2003.

\bibitem[NS]{catalogue}
Gabriele Nebe and Neil Sloane, \emph{{A Catalogue of Lattices}},
  \url{http://www.math.rwth-aachen.de/~Gabriele.Nebe/LATTICES/index.html},
  accessed 26 July 2024.

\bibitem[S{\etalchar{+}}12]{sage}
W.\thinspace{}A. Stein et~al., \emph{{S}age {M}athematics {S}oftware ({V}ersion
  4.7.2)}, The Sage Development Team, 2012, {\texttt{http://www.sagemath.org}}.

\bibitem[SEVKM19]{SEVKM}
Mathieu~Dutour Sikiri{\'c}, Philippe Elbaz-Vincent, Alexander Kupers, and
  Jacques Martinet, \emph{{Voronoi complexes in higher dimensions, cohomology
  of $GL_N(\mathbf{Z})$ for $N\geq 8$ and the triviality of
  $K_8(\mathbf{Z})$}}, 2019.

\bibitem[Sou78]{Soule}
Christophe Soul{\'e}, \emph{The cohomology of {${\rm SL}_{3}({\bf Z})$}},
  Topology \textbf{17} (1978), no.~1, 1--22. \MR{0470141 (57 \#9908)}

\bibitem[Ste07]{G}
William Stein, \emph{Modular forms, a computational approach}, Graduate Studies
  in Mathematics, vol.~79, American Mathematical Society, Providence, RI, 2007,
  With an appendix by Paul E. Gunnells. \MR{2289048 (2008d:11037)}

\end{thebibliography}

\end{document}